\documentclass[12pt,reqno]{amsart}

\newcommand\version{April 1, 2022}

\usepackage{amsmath,amsfonts,amsthm,amssymb,amsxtra}
\usepackage{bbm} 

\newtheorem{theorem}{Theorem}

\newtheorem{lemma}[theorem]{Lemma}
\newtheorem{corollary}[theorem]{Corollary}

\theoremstyle{definition}

\theoremstyle{remark}

\newtheorem{remark}[theorem]{Remark}

\newcommand{\1}{\mathbbm{1}}

\renewcommand{\epsilon}{\varepsilon}

\newcommand{\loc}{{\rm loc}}

\renewcommand{\phi}{\varphi}
\newcommand{\R}{\mathbb{R}}

\newcommand{\Sph}{\mathbb{S}}

\DeclareMathOperator{\re}{Re}

\begin{document}

\title[A Hardy inequality --- \version]{An improved one-dimensional Hardy inequality}

\author{Rupert L. Frank}
\address[Rupert L. Frank]{Mathe\-matisches Institut, Ludwig-Maximilans Universit\"at M\"unchen, The\-resienstr.~39, 80333 M\"unchen, Germany, and Munich Center for Quantum Science and Technology, Schel\-ling\-str.~4, 80799 M\"unchen, Germany, and Mathematics 253-37, Caltech, Pasa\-de\-na, CA 91125, USA}
\email{r.frank@lmu.de}

\author{Ari Laptev}
\address[Ari Laptev]{Department of Mathematics, Imperial College London, Huxley Building, 180 Queen’s Gate, London SW7 2AZ, United Kingdom, and Sirius Mathematics Center, Sirius University of Science and Technology, 1 Olympic Ave, 354340, Sochi, Russia}
\email{a.laptev@imperial.ac.uk}

\author{Timo Weidl}
\address[Timo Weidl]{Institut f\"ur Analysis, Dynamik und Modellierung, Universit\"at Stutt\-gart, Pfaffenwaldring 57, 70569 Stuttgart, Germany}
\email{weidl@mathematik.uni-stuttgart.de}

\renewcommand{\thefootnote}{${}$} \footnotetext{\copyright\, 2022 by the authors. This paper may be reproduced, in its entirety, for non-commercial purposes.\\
	Partial support through US National Science Foundation grants DMS-1954995 (R.L.F.), as well as through the Deutsche Forschungsgemeinschaft (DFG, German Research Foundation) through Germany’s Excellence Strategy EXC-2111-390814868 (R.L.F.) is acknowledged.\\
The authors wish to express their thanks to Simon Larson for discussions on the topic of this paper.}

\begin{abstract}
	We prove a one-dimensional Hardy inequality on the halfline with sharp constant, which improves the classical form of this inequality. As a consequence of this new inequality we can rederive known doubly weighted Hardy inequalities. Our motivation comes from the theory of Schr\"odinger operators and we explain the use of Hardy inequalities in that context.
\end{abstract}

\dedicatory{Dedicated, in admiration, to V.\ Maz'ya on the occasion of his 85th birthday}

\maketitle

\section{Introduction}

The celebrated Hardy inequality states that, if $1<p<\infty$ and if $u$ is a locally absolutely continuous function on $(0,\infty)$ with $\liminf_{r\to 0} |u(r)|=0$, then
\begin{equation}
	\label{eq:hardyintro}
	\int_0^\infty \frac{|u(r)|^p}{r^p}\,dr \leq \left(\frac{p}{p-1} \right)^p \int_0^\infty |u'(r)|^p\,dr \,.
\end{equation}
The constant on the right side is best possible. For background on this inequality and its generalizations we refer, for instance, to \cite{OpKu,Da,KuMaPe}.

Our basic result in this paper is the following improvement of \eqref{eq:hardyintro}.

\begin{theorem}\label{mainintro}
	Let $1<p<\infty$. Then, for any locally absolutely continuous function $u$ on $(0,\infty)$ with $\liminf_{r\to 0} |u(r)|=0$,
	\begin{equation}
		\label{eq:hardyintromain}
		\int_0^\infty \max\left\{ \sup_{0< s\leq r} \frac{|u(s)|^p}{r^p}, \sup_{r\leq s<\infty} \frac{|u(s)|^p}{s^p} \right\} dr \leq \left(\frac{p}{p-1} \right)^p \int_0^\infty |u'(r)|^p\,dr \,.
	\end{equation}
\end{theorem}

Inequality \eqref{eq:hardyintromain} is clearly an improvement of \eqref{eq:hardyintro} since
$$
\max\left\{ \sup_{0< s\leq r} \frac{|u(s)|^p}{r^p}, \sup_{r\leq s<\infty} \frac{|u(s)|^p}{s^p} \right\} \geq  \frac{|u(r)|^p}{r^p} \,.
$$
Remarkably, the constant in \eqref{eq:hardyintromain} is the same as that in \eqref{eq:hardyintro}. 

Surprisingly, given how natural \eqref{eq:hardyintromain} is, we have not been able to locate an earlier occurence in the literature. In the case $p=2$ it appeared recently in our book \cite{FrLaWe}. Here we show that the same proof extends to arbitrary $p$; see Section \ref{sec:main}. The proof will use \eqref{eq:hardyintro} as an ingredient, together with a simple rearrangement argument.

In the remaining sections of this paper, we will draw some conclusions from \eqref{eq:hardyintromain}. Let us summarize the most important ones. Indeed, bounding the maximum in \eqref{eq:hardyintromain} by either one of the two quantities, we arrive at the two inequalities
\begin{equation}
	\label{eq:hardyintromain1}
	\int_0^\infty \sup_{0< s\leq r} \frac{|u(s)|^p}{r^p} \,dr \leq \left(\frac{p}{p-1} \right)^p \int_0^\infty |u'(r)|^p\,dr
\end{equation}
and
\begin{equation}
	\label{eq:hardyintromain2}
	\int_0^\infty \sup_{r\leq s<\infty} \frac{|u(s)|^p}{s^p} \,dr \leq \left(\frac{p}{p-1} \right)^p \int_0^\infty |u'(r)|^p\,dr \,,
\end{equation}
valid for the same class of functions $u$ as in Theorem \ref{mainintro}. Again, we have not found \eqref{eq:hardyintromain1} and \eqref{eq:hardyintromain2} stated explicitly in the literature. We will show here, however, that they are equivalent to certain inequalities that are known. Specifically,  \eqref{eq:hardyintromain1} and \eqref{eq:hardyintromain2} are equivalent, respectively, to the following two weighted inequalities,
\begin{equation}
	\label{eq:hardyintromainweight1}
	\int_0^\infty W(r) |u(r)|^p \,dr \leq \frac{p^p}{(p-1)^{p-1}} 
	\left( \sup_{s>0} s^{p-1} \int_s^\infty W(t)\,dt \right) \int_0^\infty |u'(r)|^p\,dr
\end{equation}
and
\begin{equation}
	\label{eq:hardyintromainweight2}
	\int_0^\infty W(r) |u(r)|^p \,dr \leq \left(\frac{p}{p-1} \right)^p 
	\left( \sup_{s>0} s^{-1} \int_0^s W(t) t^p \,dt \right) \int_0^\infty |u'(r)|^p\,dr \,,
\end{equation}
valid for all nonnegative, measurable functions $W$ on $(0,\infty)$ and all $u$ as above. Note that for $W(r) = r^{-p}$, \eqref{eq:hardyintromainweight1} and \eqref{eq:hardyintromainweight2} both reduce to \eqref{eq:hardyintro}. Inequality \eqref{eq:hardyintromainweight1} for $p=2$ is due to Kac and Kre\u{\i}n \cite{KaKr}. We review their proof in Section \ref{sec:direct} and supplement it with a direct proof of \eqref{eq:hardyintromainweight2} for $p=2$. Returning to general $p$, both inequalities \eqref{eq:hardyintromainweight1} and \eqref{eq:hardyintromainweight2} are special cases of a more general family of inequalities due to Tomaselli \cite{To}, which we discuss momentarily.

In order to show that inequalities \eqref{eq:hardyintromain1} and \eqref{eq:hardyintromain2} are equivalent to \eqref{eq:hardyintromainweight1} and \eqref{eq:hardyintromainweight2}, we employ a duality argument. This is presented in Section \ref{sec:duality}.

Next, by a well-known change of variables argument, we see that \eqref{eq:hardyintromainweight1} and \eqref{eq:hardyintromainweight2} are equivalent to the following doubly weighted inequalities.

\begin{theorem}\label{weightd}
	Let $1<p<\infty$ and let $V,W$ be nonnegative, a.e.-finite, measurable functions on $(0,\infty)$ such that
	$$
	\int_0^s V(t)^{-\frac{1}{p-1}}\,dt <\infty
	\qquad\text{for all}\ s\in(0,\infty) \,.
	$$
	Then, for any locally absolutely continuous function $u$ on $(0,\infty)$ with $\liminf_{r\to 0} |u(r)|\!=0$,
	\begin{equation}
		\label{eq:hardyintromainweightd1}
		\int_0^\infty W(r) |u(r)|^p \,dr \leq \frac{p^p}{(p-1)^{p-1}} 
		\overline B \int_0^\infty V(r) |u'(r)|^p\,dr
	\end{equation}
	and
	\begin{equation}
		\label{eq:hardyintromainweightd2}
		\int_0^\infty W(r) |u(r)|^p \,dr \leq \left(\frac{p}{p-1} \right)^p \underline B \int_0^\infty V(r) |u'(r)|^p\,dr
	\end{equation}
	with
	\begin{equation}
		\label{eq:overb}
		\overline B := \sup_{s>0} \left( \int_0^s V(t)^{-\frac1{p-1}}\,dt \right)^{p-1} \left( \int_s^\infty W(t)\,dt \right)
	\end{equation}
	and
	\begin{equation}
		\label{eq:underb}
		\underline B := \sup_{s>0} \left( \int_0^s V(t)^{-\frac{1}{p-1}}dt \right)^{-1} \int_0^s W(t) \left( \int_0^t V(t')^{-\frac{1}{p-1}}dt' \right)^p dt \,.
	\end{equation}
\end{theorem}

Note that for $V\equiv 1$, \eqref{eq:hardyintromainweightd1} and \eqref{eq:hardyintromainweightd2} reduce to \eqref{eq:hardyintromainweight1} and \eqref{eq:hardyintromainweight2}, respectively.

For the sake of completeness, we will provide in Section \ref{sec:cov} the details of the change of variables argument that proves the equivalence of \eqref{eq:hardyintromainweightd1} and \eqref{eq:hardyintromainweightd2} with \eqref{eq:hardyintromainweight1} and \eqref{eq:hardyintromainweight2}, respectively. In that section we will also recall the well-known fact that the validity of inequalities \eqref{eq:hardyintromainweightd1} and \eqref{eq:hardyintromainweightd2} with \emph{some} constant implies that finiteness of $\overline B$ and $\underline B$ as defined in \eqref{eq:overb} and \eqref{eq:underb}. This implies, in particular, that $\overline B$ is finite if and only if $\underline B$ is finite, and that they are comparable.

Let us discuss the history of Theorem \ref{weightd}. Both inequalities \eqref{eq:hardyintromainweightd1} and \eqref{eq:hardyintromainweightd2} appear in Tomaselli's paper \cite{To}, see equations (9) and (10) there; see also equations (27') and (27'') in the exposition \cite{Ta} of Tomaselli's work. An independent, very elegant proof of inequality \eqref{eq:hardyintromainweightd1} was given by Muckenhoupt \cite{Mu}. Apparently, most of the relevant textbooks put their focus on inequality \eqref{eq:hardyintromainweightd1} and gloss over \eqref{eq:hardyintromainweightd2}; see, for instance, \cite[Theorem 1.3.2/1]{Ma} and \cite[Theorem 1.14]{OpKu}. Inequality \eqref{eq:hardyintromainweightd2} is only briefly mentioned without proof in \cite[Section 2.8]{OpKu}. Tomaselli's proof of both inequalities is based on the analysis of certain ordinary differential equations, while Muckenhoupt's proof of \eqref{eq:hardyintromainweightd1} is based on Minkowski's inequality. It is not clear to us whether this latter method of proof can be used to establish \eqref{eq:hardyintromainweightd2}. Our proof of Theorem \ref{weightd} via Theorem \ref{mainintro} and duality seems to be new.

As mentioned before, the constants $\overline B$ and $\underline B$ are comparable. In this comparability, however, some constants appear, which destroy the optimality of the inequalities and which are not acceptable in certain applications where these constants matter. From the point of view of these applications, inequalities \eqref{eq:hardyintromainweightd1} and \eqref{eq:hardyintromainweightd2} are not equivalent and one might be better suited in one problem and one in another. The main difference between the two inequalities is that in the constant $\overline B$, $W$ is integrated near infinity, whereas in $\underline B$ it is integrated near the origin. This difference is crucial in applications and was the main motivation of our study.

To be more specific, our application concerns sharp conditions on the potential for a Schr\"odinger operator to have only a finite number of negative eigenvalues. We describe this in more detail in Section \ref{sec:so}. The relevance of Hardy inequalities for this kind of questions was emphasized by Birman \cite{Bi}.

To summarize our discussion so far, we have seen that the doubly weighted Hardy inequalities \eqref{eq:hardyintromainweightd1} and \eqref{eq:hardyintromainweightd2} are equivalent to inequalities \eqref{eq:hardyintromain1} and \eqref{eq:hardyintromain2}, respectively. Moreover, \eqref{eq:hardyintromain1} and \eqref{eq:hardyintromain2} are both consequences of the new inequality in Theorem \ref{mainintro}.

A natural question, which we have not been able to answer, is to find a doubly weighted Hardy inequality that is equivalent to \eqref{eq:hardyintromain}. This is an \emph{open problem}.

All inequalities that we have considered so far were for locally absolutely continuous functions $u$ on $(0,\infty)$ with $\liminf_{r\to 0}|u(r)|=0$. For the sake of completeness, let us also state the inequalities for locally absolutely continuous functions $u$ on $(0,\infty)$ with $\liminf_{r\to \infty}|u(r)|=0$. They are deduced from the former ones by the change of variables $r\mapsto r^{-1}$ and read as follows.

\begin{theorem}\label{weightdinv}
	Let $1<p<\infty$ and let $V,W$ be nonnegative, a.e.-finite, measurable functions on $(0,\infty)$ such that
	$$
	\int_s^\infty V(t)^{-\frac{1}{p-1}}dt <\infty
	\qquad\text{for all}\ s\in (0,\infty) \,.
	$$
	Then for any locally absolutely continuous function $u$ on $(0,\infty)$ with $\liminf_{r\to \infty}|u(r)|\!=0$,
	\begin{equation}
		\label{eq:hardyintromainweightdinv1}
		\int_0^\infty W(r) |u(r)|^p \,dr \leq \frac{p^p}{(p-1)^{p-1}} 
		\overline B' \int_0^\infty V(r) |u'(r)|^p\,dr
	\end{equation}
	and
	\begin{equation}
		\label{eq:hardyintromainweightdinv2}
		\int_0^\infty W(r) |u(r)|^p \,dr \leq \left(\frac{p}{p-1} \right)^p \underline B' \int_0^\infty V(r) |u'(r)|^p\,dr \,,
	\end{equation}
	with
	\begin{equation}
		\overline B' := \sup_{s>0} \left( \int_s^\infty V(t)^{-\frac1{p-1}}\,dt \right)^{p-1} \left( \int_0^s W(t)\,dt \right)
	\end{equation}
	and
	\begin{equation}
		\underline B' := \sup_{s>0} \left( \int_s^\infty V(t)^{-\frac{1}{p-1}}dt \right)^{-1} \int_s^\infty W(t) \left( \int_t^\infty V(t')^{-\frac{1}{p-1}}dt' \right)^p dt \,.
	\end{equation}
\end{theorem}

Note that the roles of $0$ and $\infty$ in the definitions of $\overline B'$ and $\underline B'$ have changed relative to $\overline B$ and $\underline B$, but so has the point where the `boundary condition' on $u$ is imposed.

In the remainder of this paper, we will provide the proofs of the claims made in this introduction.

It is with great admiration and respect that we dedicate this paper to V. Maz'ya, who has shaped our understanding of Sobolev spaces and Schr\"odinger operators.


\section{The main inequality}\label{sec:main}

In this section we prove our main result, Theorem \ref{mainintro}. It will be somewhat more convenient to work with Hardy inequalities in an equivalent integral rather than differential form.

\begin{theorem}\label{mainintegral}
	Let $1<p<\infty$. Then, for any $f\in L^p(0,\infty)$,
	\begin{equation}
		\label{eq:main}
		\int_0^\infty \sup_{0<s<\infty} \left| \min\left\{ \frac1r,\frac1s \right\} \int_0^s f(t)\,dt \right|^p dr \leq \left( \frac{p}{p-1}\right)^p \int_0^\infty |f(r)|^p\,dr \,.
	\end{equation}
\end{theorem}

Clearly, the correspondence $\int_0^s f(t)\,dt = u(s)$, $f(s)=u'(s)$ gives the equivalence between Theorems \ref{mainintro} and \ref{mainintegral}.

\begin{proof}
	We denote by $f^*$ the nonincreasing rearrangement of $f$. This is a nonincreasing, nonnegative function on $(0,\infty)$ such that $\{ |f|>\tau\}|=|\{f^*>\tau\}|$ for all $\tau>0$. For more on this rearrangement and, in particular, the following two simple properties that we will use, we refer, for instance, to \cite[Section 2.1]{BeSh}. On the one hand, by the equimeasurability property,
	$$
	\int_0^\infty |f(r)|^p\,dr = \int_0^\infty (f^*(r))^p\,dr \,.
	$$
	On the other hand, by the simplest rearrangment inequality, for any $s>0$,
	$$
	\left| \int_0^s f(t)\,dt \right| \leq \int_0^s |f(t)|\,dt \leq \int_0^s f^*(t)\,dt \,.
	$$
	Thus, for any $r>0$,
	$$
	\sup_{0<s<\infty} \left| \min\left\{ \frac1r,\frac1s \right\} \int_0^s f(t)\,dt \right|
	\leq  \sup_{0<s<\infty} \min\left\{ \frac1r,\frac1s \right\} \int_0^s f^*(t)\,dt \,.
	$$
	As a consequence, if we can prove the inequality for $f^*$, it holds also for $f$.
	
	The advantage of $f^*$ is that the supremum can be computed. Indeed, since $f^*$ is nonincreasing, we have for all $r\leq s$ and all $t>0$, $f^*(t)\leq f^*(rt/s)$, so
	$$
	\frac1s \int_0^s f^*(t)\,dt \leq \frac1s \int_0^s f^*(rt/s)\,dt = \frac1r \int_0^r f^*(u)\,du \,.
	$$
	Thus,
	$$
	\sup_{r\leq s<\infty} \frac{1}{s} \int_0^s f^*(t)\,dt = \frac{1}{r} \int_0^r f^*(t)\,dt
	$$
	and, therefore,
	$$
	\sup_{0<s<\infty} \min\left\{ \frac1r,\frac1s \right\} \int_0^s f^*(t)\,dt = \frac 1r \int_0^r f^*(t)\,dt \,.
	$$
	Thus, \eqref{eq:main} for $f^*$ follows from the standard Hardy's inequality \eqref{eq:hardyintro} or, more precisely, its equivalent integral form.
\end{proof}


\section{A duality result}\label{sec:duality}

Our goal in this section is to show that the Hardy inequalities \eqref{eq:hardyintromain1} and \eqref{eq:hardyintromain2} are equivalent to the weighted Hardy inequalities \eqref{eq:hardyintromainweight1} and \eqref{eq:hardyintromainweight2}. As mentioned in the introduction, this argument relies on a duality result, which we now state and prove.

For parameters $\alpha,\beta>0$ and for nonnegative, a.e.-finite, measurable functions $f,g$ on $(0,\infty)$ we set
\begin{align*}
	\underline{\mu}_{\alpha}(f) & := \int_0^\infty\text{ess-sup}_{0<s\leq r} f(s)\,\frac{dr}{r^{1+\alpha}} \,, &
	\overline{\mu}_{\beta}(f) & := \int_0^\infty\text{ess-sup}_{r\leq s<\infty} f(s)\,\frac{dr}{r^{1-\beta}} \,, \\
	\overline{\nu}_\alpha(g) &:= \text{sup}_{r>0} r^{\alpha} \int_r^\infty g(s) \,ds \,, &
	\underline{\nu}_\beta(g) & := \text{sup}_{r>0} r^{-\beta} \int_0^r g(s) \,ds \,.
\end{align*}

\begin{theorem}\label{dual}
	Let $\alpha,\beta>0$. Then, for any nonnegative, measurable function $f$ on $(0,\infty)$,
	\begin{equation}
		\label{eq:dual1}
		\sup\left\{ \int_0^\infty fg\,dr:\ g\geq 0 \,,\ \overline\nu_\alpha(g) \leq 1 \right\} = \alpha\, \underline\mu_\alpha(f)
	\end{equation}
	and
	\begin{equation}
		\label{eq:dual2}
		\sup\left\{ \int_0^\infty fg\,dr:\ g\geq 0\,,\ \underline\nu_\beta(g) \leq 1 \right\} = \beta\, \overline\mu_\beta(f) \,,
	\end{equation}
	and conversely, for any nonnegative, measurable function $g$ on $(0,\infty)$,
	\begin{equation}
		\label{eq:dual3}
		\sup\left\{ \int_0^\infty fg\,dr:\ f\geq 0\,,\ \underline\mu_\alpha(f) \leq 1 \right\} = \alpha\, \overline\nu_\alpha(g)
	\end{equation}
	and
	\begin{equation}
		\label{eq:dual4}
		\sup\left\{ \int_0^\infty fg\,dr:\ f\geq 0\,,\ \overline\mu_\beta(f) \leq 1 \right\} = \beta\, \underline\nu_\beta(g) \,.
	\end{equation}
\end{theorem}

This result for $\beta=1$ is closely related to a result of Luxemburg and Zaanen \cite{LuZa}; see also \cite[Exercise 1.6]{BeSh}.

\begin{proof}
	\emph{Step 1a.}
	We set $\underline{f}(r):= \text{ess-sup}_{0<s\leq r} f(s)$ and note that $f\leq\underline f$ almost everywhere. (For a careful proof of this fact one can proceed as in \cite[Lemma 4.2]{LuZa}.) Therefore, we can bound
	$$
	\int_0^\infty f(r)g(r) \,dr \leq \int_0^\infty \underline f(r) g(r)\,dr \,.
	$$
	Since $\underline f$ is nondecreasing, for each $\tau>0$, the set $\{\underline f>\tau\}$ is an interval of the form $(a_\tau,\infty)$. We bound
	\begin{align*}
		\int_0^\infty \1_{\{\underline f >\tau\}}(r) g(r)\,dr & = \int_{a_\tau}^\infty g(r) \,dr \leq a_\tau^{-\alpha} \sup_{s>0} s^{\alpha} \int_s^\infty g(r)\,dr \\
		& = \alpha \int_0^\infty \1_{\{\underline f >\tau\}}(r)\,\frac{dr}{r^{1+\alpha}}\ \overline\nu_\alpha(g) \,.
	\end{align*}
	Integrating this inequality with respect to $\tau$, we arrive at
	$$
	\int_0^\infty \underline f(r) g(r)\,dr \leq \alpha \int_0^\infty \underline f(r)\,\frac{dr}{r^{1+\alpha}}\  \overline\nu_\alpha(g) = \alpha\, \underline\mu_\alpha(f)\, \overline\nu_\alpha(g) \,.
	$$
	Thus, we have shown that
	$$
	\int_0^\infty f(r)g(r) \,dr \leq \alpha\, \underline\mu_\alpha(f)\, \overline\nu_\alpha(g) \,.
	$$
	This proves $\leq$ in \eqref{eq:dual1} and \eqref{eq:dual3}.
	
	\medskip
	
	\emph{Step 1b.} 
	We set $\overline{f}(r):= \text{ess-sup}_{r\leq s<\infty} f(s)$ and bound
	$$
	\int_0^\infty f(r)g(r) \,dr \leq \int_0^\infty \overline f(r) g(r)\,dr \,.
	$$
	Since $\overline f$ is nonincreasing, for each $\tau>0$, the set $\{\overline f>\tau\}$ is an interval of the form $(0,b_\tau)$. We bound
	\begin{align*}
		\int_0^\infty \1_{\{\overline f >\tau\}}(r) g(r)\,dr & = \int_0^{b_\tau} g(r) \,dr \leq b_\tau^\beta \sup_{s>0} s^{-\beta} \int_0^s g(r)\,dr \\
		& = \beta \int_0^\infty \1_{\{\overline f >\tau\}}(r)\,\frac{dr}{r^{1-\beta}} \ \underline\nu_\beta(g) \,.
	\end{align*}
	Integrating this inequality with respect to $\tau$, we arrive at
	$$
	\int_0^\infty \overline f(r) g(r)\,dr \leq \beta \int_0^\infty \overline f(r)\, \frac{dr}{r^{1-\beta}}\  \underline\nu_\beta(g) = \beta\, \overline\mu_\beta(f)\, \underline\nu_\beta(g) \,.
	$$
	Thus, we have shown that
	$$
	\int_0^\infty f(r)g(r) \,dr \leq \beta\, \overline\mu_\beta(f)\, \underline\nu_\beta(g) \,.
	$$
	This proves $\leq$ in \eqref{eq:dual2} and \eqref{eq:dual4}.
	
	\medskip
	
	\emph{Step 2a.}
	For $s>0$, we set $f_s = \alpha s^{\alpha} \1_{(s,\infty)}$ and note that $\underline\mu_\alpha(f_s)=1$ and
	$$
	\sup_{s>0} \int_0^\infty f_s g\,dr = \alpha\, \overline\nu_\alpha(g) \,.
	$$
	Restricting the supremum in \eqref{eq:dual3} to $f_s$, $s>0$, we obtain $\geq$ in \eqref{eq:dual3}.
	
	\medskip
	
	\emph{Step 2b.}
	For $s>0$, we set $f_s = \beta s^{-\beta} \1_{(0,s)}$ and note that $\overline\mu_\beta(f_s)=1$ and
	$$
	\sup_{s>0} \int_0^\infty f_s g\,dr = \beta\, \underline\nu_\beta(g) \,.
	$$
	Restricting the supremum in \eqref{eq:dual4} to $f_s$, $s>0$, we obtain $\geq$ in \eqref{eq:dual4}.
	
	\medskip
	
	\emph{Step 3a.}
	Steps 1a and 2a show that we have equality in \eqref{eq:dual3} and $\leq$ in \eqref{eq:dual1}. Equality in \eqref{eq:dual1} now follows from a duality theorem \cite[Theorem 71.1]{Za}, because $\overline\mu_\alpha$ satisfies the Fatou property, that is,
	\begin{equation}
		\label{eq:fatou}
		0\leq f_n \uparrow f \ \text{a.e.}
		\qquad\implies\qquad
		\underline\mu_\alpha(f_n) \uparrow \underline\mu_\alpha(f) \,.
	\end{equation}
	Indeed, the assumption in \eqref{eq:fatou} implies that $\text{ess-sup}_{0<s\leq r} f_n(s)\to \text{ess-sup}_{0<s\leq r} f(s)$ for a.e. $r>0$, and then the conclusion in \eqref{eq:fatou} follows by monotone convergence.
	
	Since the reference \cite{Za} may not be easily accessible (and even less so the reference for \cite[Theorem 2.1]{LuZa}), it might be preferrable to appeal to \cite[Theorem 2.7]{BeSh}. There, in contrast to \cite{Za}, the Fatou property is included in the definition of a function norm; see property (P3) in \cite[Definition 1.1.1]{BeSh}. The latter definition also includes properties (P4) and (P5), which are not necessarily valid in our case (for instance, $\underline\mu_\alpha(\1_{(0,1)})=\infty$). This is not a problem, however, because if one follows the proof of \cite[Theorem 2.7]{BeSh} in our case, one sees that the latter two properties are not really needed. (Indeed, instead of only requiring the sets $R_N$ in that proof to have finite measure, one chooses them to be compact subsets of $(0,\infty)$. The remainder goes through without changes.)
		
	\medskip
	
	\emph{Step 3b.}
	The validity of \eqref{eq:dual2} and \eqref{eq:dual4} is deduced from Steps 1b and 2b in complete analogy with Step 3a.
\end{proof}

\begin{remark}
	We note that the assertion of Theorem \ref{dual} for one single value of either $\alpha$ or $\beta$ implies the assertion for all values of $\alpha$ and $\beta$. This follows by a change of variables, considering $\tilde f(\rho)=f(\rho^\gamma)$ and $\tilde g(\rho) = g(\rho^\gamma) \rho^{\gamma-1}$ for a suitably chosen $\gamma$.
\end{remark}

Let us now apply Theorem \ref{dual} to prove the equivalence between \eqref{eq:hardyintromain1} and \eqref{eq:hardyintromainweight1} and that between \eqref{eq:hardyintromain2} and \eqref{eq:hardyintromainweight2}.

To see that \eqref{eq:hardyintromain1} implies \eqref{eq:hardyintromainweight1}, we note that, by either \eqref{eq:dual1} or \eqref{eq:dual3},
$$
\int_0^\infty W|u|^p\,dr \leq (p-1)\, \underline\mu_{p-1}(|u|^p)\, \overline\nu_{p-1}(W)
$$
and bound $\underline\mu_{p-1}(|u|^p)$ using \eqref{eq:hardyintromain1} to arrive at \eqref{eq:hardyintromainweight1}. Conversely, by \eqref{eq:dual1},
$$
(p-1)\, \mu_{p-1}(|u|^p) = \sup\left\{ \int_0^\infty W|u|^p\,dr :\ \overline\nu_{p-1}(W)\leq 1 \right\}
$$
and we bound the right side using \eqref{eq:hardyintromainweight1} to arrive at \eqref{eq:hardyintromain1}.

The equivalence of \eqref{eq:hardyintromain2} and \eqref{eq:hardyintromainweight2} is similar, except that we use \eqref{eq:dual2} and \eqref{eq:dual4} with $f=r^{-p} |u|^p$, $g=r^p W$ and $\beta =1$.

We conclude this section by commenting on the \emph{open problem} mentioned in the introduction, namely that of finding a weighted inequality that is equivalent to \eqref{eq:hardyintromain}. In other words, given a nonnegative, measurable function $g$ on $(0,\infty)$, we would like to compute the quantity
$$
\sup\left\{ \int_0^\infty fg\,dr :\ f\geq 0 \,,\ \int_0^\infty \max\left\{ \sup_{0< s\leq r} \frac{f(s)}{r^p}, \sup_{r\leq s<\infty} \frac{f(s)}{s^p} \right\} dr \leq 1  \right\}.
$$
This would probably give rise to a weighted Hardy inequality that implies both \eqref{eq:hardyintromainweight1} and \eqref{eq:hardyintromainweight2}.


\section{The doubly-weighted Hardy inequality}\label{sec:cov}

In this section we discuss the inequalities in Theorems \ref{weightd} and \ref{weightdinv}. The material is well known, see, for instance, \cite{To,Ta,OpKu}, and is included here only to provide some context and to make this paper accessible to nonexperts.

First, we show that the inequalities \eqref{eq:hardyintromainweight1} and \eqref{eq:hardyintromainweight2} with a single weight function $W$ imply inequalities \eqref{eq:hardyintromainweightd1} and \eqref{eq:hardyintromainweightd2}. The converse is obvious, by taking $V\equiv 1$.

To shorten the statements, let us denote by $\mathcal F$ the set of locally absolutely continuous function $u$ on $(0,\infty)$ with $\liminf_{r\to 0}|u(r)|=0$, by $\mathcal W$ the set of all nonnegative, a.e.-finite, measurable functions on $(0,\infty)$ and by $\mathcal V_p$ the subset of $V\in\mathcal W$ such that
$$
\int_0^s V(t)^{-\frac{1}{p-1}}\,dt <\infty
\qquad\text{for all}\ s\in(0,\infty) \,.
$$

\begin{lemma}\label{cov}
	Let $1<p<\infty$.\\
	(1) Assume that there is a constant $c<\infty$ such that for all $W\in\mathcal W$ and $u\in\mathcal F$ one has
	$$
	\int_0^\infty W(r)|u(r)|^p\,dr \leq c \left( \sup_{s>0} s^{p-1} \int_s^\infty W(t)\,dt \right) \int_0^\infty |u'(r)|^p\,dr \,.
	$$
	Then for all $W\in\mathcal W$, $V\in\mathcal V_p$ and $u\in\mathcal F$ one has
	$$
	\int_0^\infty W(r)|u(r)|^p\,dr \leq c\, \overline B_{V,W}\, \int_0^\infty V(r) |u'(r)|^p\,dr \,.
	$$
	with $\overline B_{V,W} = \overline B$ from \eqref{eq:overb}.\\
	(2) Assume that there is a constant $c<\infty$ such that for all $W\in\mathcal W$ and $u\in\mathcal F$ one has
	$$
	\int_0^\infty W(r)|u(r)|^p\,dr \leq c \left( \sup_{s>0} s^{-1} \int_0^s W(t) t^p\,dt \right) \int_0^\infty |u'(r)|^p\,dr \,.
	$$
	Then for all $W\in\mathcal W$, $V\in\mathcal V_p$ and $u\in\mathcal F$ one has
	$$
	\int_0^\infty W(r)|u(r)|^p\,dr \leq c\, \underline B_{V,W}\, \int_0^\infty V(r) |u'(r)|^p\,dr \,.
	$$
	with $\underline B_{V,W} = \underline B$ from \eqref{eq:underb}.	
\end{lemma}

\begin{proof}
	Let $V\in\mathcal V_p$. Then the function $\phi$ on $(0,\infty)$, defined by
	$$
	\phi(r) := \int_0^r V(s)^{-\frac{1}{p-1}}\,ds
	\qquad\text{for all}\ r\in(0,\infty) \,,
	$$
	is strictly increasing (since $V$ is a.e.-finite) and locally absolutely continuous (by the integrability assumption on $V$). Its inverse function $\psi$ is defined on $(0,L)$ with $L:=\int_0^\infty V(s)^{-\frac{1}{p-1}}\,ds\in(0,\infty)\cup\{\infty\}$ and is strictly increasing. Since the set where $\phi' = V^{-\frac1{p-1}}$ vanishes has measure zero, $\psi$ is locally absolutely continuous in $(0,L)$ \cite[Exercise 3.46]{Le}. Using the chain rule \cite[Corollary 3.66]{Le} one deduces that
	$$
	\psi'(\rho) = V(\psi(\rho))^{\frac{1}{p-1}}
	\qquad\text{for a.e.}\ \rho\in(0,L) \,.
	$$
		
	If $u\in\mathcal F$, then, by \cite[Exercise 3.67]{Le}, $\tilde u = u\circ \psi$ is locally absolutely continuous in $(0,L)$ and, by the chain rule \cite[Corollary 3.66]{Le},
	$$
	\tilde u'(\rho) = u'(\psi(\rho))\,\psi'(\rho) = V(\psi(\rho))^{\frac{1}{p}} u'(\psi(\rho))  \psi'(\rho)^\frac{1}{p}
	\qquad\text{for a.e.}\ \rho\in(0,L) \,.
	$$
	Thus, by the change of variables formula \cite[Corollary 3.78]{Le},
	$$
	\int_0^L |\tilde u'(\rho)|^p\,d\rho = \int_0^L V(\psi(\rho)) |u'(\psi(\rho))|^p \psi'(\rho)\,d\rho = \int_0^\infty V(r)|u'(r)|^p\,dr \,.
	$$
	We assume that the right side is finite (for otherwise there is nothing to prove). If $L<\infty$, then the finiteness of the left side implies that $\tilde u$ extends continuously to the point $r=L$ \cite[Remark 2.7]{FrLaWe} and we can extend $\tilde u$ continuously by a constant to $[L,\infty)$.
	
	Now let $W\in\mathcal W$ and define a function $\tilde W$ on $(0,\infty)$ by
	$$
	\tilde W(\rho) := W(\psi(\rho)) \psi'(\rho)
	\qquad\text{for}\ \rho\in (0,L)
	$$
	and, if $L<\infty$, by $\tilde W(\rho):=0$ for $\rho\in[L,\infty)$. We note that, again by the change of variables formula \cite[Corollary 3.78]{Le},
	$$
	\int_0^\infty \tilde W(\rho) |\tilde u(\rho)|^p\,d\rho = \int_0^L W(\psi(\rho)) |u(\psi(\rho))|^p \psi'(\rho)\,d\rho = \int_0^\infty W(r) |u(r)|^p \,dr \,.
	$$
	
	The claimed doubly-weighted Hardy inequality is a consequence of the assumed single-weighted inequality for $\tilde u$ with weigth $\tilde W$. It remains to see how the constants transform. We have, for $\sigma<L$,
	$$
	\sigma^{p-1} \int_\sigma^\infty \tilde W(\tau)\,d\tau = \sigma^{p-1} \int_\sigma^L W(\psi(\tau))\psi'(\tau)\,d\tau = \left( \phi(\psi(\sigma)) \right)^{p-1} \int_{\psi(\sigma)}^\infty W(t)\,dt
	$$
	and, similarly,
	$$
	\sigma^{-1} \int_0^\sigma \tilde W(\tau) \tau^p\,d\tau = \sigma^{-1} \int_0^\sigma W(\psi(\tau)) \tau^p \psi'(\tau) \,d\tau = \left( \phi(\psi(\sigma)) \right)^{-1} \int_0^{\psi(\sigma)} W(t) \phi(t)^p \,dt \,.
	$$
	Recalling the definition of $\phi$, we arrive at the quantities $\overline B_{V,W}$ and $\underline B_{V,W}$. This completes the proof.
\end{proof}

Our next result shows that both conditions $\overline B<\infty$ and $\underline B<\infty$ are necessary for the doubly-weighted Hardy inequality to hold.

\begin{lemma}\label{converse}
	Let $1<p<\infty$ and let $W\in\mathcal W$ and $V\in\mathcal V_p$. Assume that there is a $C<\infty$ such that for any $u\in\mathcal F$,
	$$
	\int_0^\infty W(r)|u(r)|^p\,dr \leq C \int_0^\infty V(r) |u'(r)|^p\,dr \,.
	$$
	Then
	$$
	\overline B \leq C
	\qquad\text{and}\qquad
	\underline B \leq C \,.
	$$
\end{lemma}

\begin{proof}
	For fixed $s>0$, we evaluate the assumed Hardy inequality for the function
	$$
	u(r) = \int_0^r \1_{(0,s)}(t) V(t)^{-\frac{1}{p-1}}\,dt
	$$
	and obtain
	\begin{align*}
		& \left( \int_0^s V(t)^{-\frac{1}{p-1}}dt \right)^{-1} \int_0^s W(t) \left( \int_0^t V(t')^{-\frac{1}{p-1}}dt' \right)^p dt \\
		& + \left( \int_0^s V(t)^{-\frac{1}{p-1}}\,dt \right)^{p-1} \int_s^\infty W(t)\,dt \leq C
		\qquad\text{for all}\ s\in(0,\infty) \,.
	\end{align*}
	Dropping one of the two terms on the left side and taking the supremum over $s\in(0,\infty)$, we obtain the claimed inequalities.
\end{proof}

\begin{remark}
	Note that Lemma \ref{converse}, together with Theorem \ref{weightd}, implies that
	$$
	\overline B \leq \left( \frac{p}{p-1} \right)^p \, \underline B
	\qquad\text{and}\qquad
	\underline B \leq \frac{p^p}{(p-1)^{p-1}}\, \overline B \,.
	$$
	We do \emph{not} claim that the constants here are sharp.
\end{remark}


\section{Proof of Theorem \ref{weightd} for $p=2$}\label{sec:direct}

In this section we provide a proof of Theorem \ref{weightd} in the special case $p=2$ that does not rely on Theorem \ref{mainintro}. Just like the proof of Theorem \ref{mainintro}, however, it does rely on \eqref{eq:hardyintro}.

Due to the equivalence discussed in the previous section, it suffices to prove \eqref{eq:hardyintromainweight1} and \eqref{eq:hardyintromainweight2} for $p=2$. As we mentioned in the introduction, the proof of \eqref{eq:hardyintromainweight1} for $p=2$ is due to Kac and Kre\u{\i}n \cite{KaKr}; see also \cite[Proof of Theorem 2.4]{Bi} by Birman and Pavlov. Our proof of \eqref{eq:hardyintromainweight2} for $p=2$ is a small variation of theirs, which, since we have not seen it anywhere, might be worth recording. We find it instructive to present both proofs, so that the similarities and differences become clearer.

\begin{proof}[Proof of \eqref{eq:hardyintromainweight1} for $p=2$]
	Let $u$ be a locally absolutely continuous function on $(0,\infty)$ with $\liminf_{r\to 0} |u(r)|=0$. Let
	$$
	\Omega(r) := \int_r^\infty W(t)\,dt \,,
	\qquad
	C:= \sup_{r>0} r \Omega(r) \,.
	$$
	We may also assume that $u'\in L^2(0,\infty)$ and $C<\infty$, for otherwise there is nothing to prove. For any $0<\epsilon<M<\infty$ we have
	\begin{align}\label{eq:directproof1}
		\int_\epsilon^M W(r)|u(r)|^2\,dr & = - \int_\epsilon^M \Omega'(r)|u(r)|^2\,dr \notag \\
		& = 2 \re \int_\epsilon^M \Omega(r) \overline{u(r)} u'(r)\,dr - \Omega(M) |u(M)|^2 + \Omega(\epsilon) |u(\epsilon)|^2 \notag \\
		& \leq 2 \re \int_\epsilon^M \Omega(r) \overline{u(r)} u'(r)\,dr + \Omega(\epsilon) |u(\epsilon)|^2 \notag \\
		& \leq C \left( 2 \int_\epsilon^M \frac{|u(r)|}{r} |u'(r)|\,dr + \epsilon^{-1} |u(\epsilon)|^2 \right). 
	\end{align}
	Using \eqref{eq:hardyintro}, we can bound
	\begin{align*}
	2 \int_\epsilon^M \frac{|u(r)|}{r} |u'(r)|\,dr
	& \leq 2 \left( \int_0^\infty \frac{|u(r)|^2}{r^2}\,dr \right)^{1/2} \left( \int_0^\infty |u'(r)|^2\,dr \right)^{1/2} \\
	& \leq 4 \int_0^\infty |u'(r)|^2\,dr \,. 
	\end{align*}
	Inserting this into \eqref{eq:directproof1}, we see that it remains to control $\epsilon^{-1} |u(\epsilon)|^2$. For $0<\rho<\epsilon$, we have
	$$
	|u(\epsilon)| \leq |u(\rho)| + \int_\rho^\epsilon |u'(r)|\,dr \leq |u(\rho)| + \epsilon^{1/2} \left( \int_0^\epsilon |u'(r)|^2\,dr \right)^{1/2} \,. 
	$$
	Choosing a sequence of $\rho$'s along which $u$ tends to zero, which exists by assumption, we deduce that
	$$
	\epsilon^{-1} |u(\epsilon)|^2 \leq \int_0^\epsilon |u'(r)|^2\,dr \,.
	$$
	By dominated convergence, this tends to zero as $\epsilon\to 0$, which concludes the proof.	
\end{proof}

\begin{proof}[Proof of \eqref{eq:hardyintromainweight2} for $p=2$]
	Let $u$ be a locally absolutely continuous function on $(0,\infty)$ with $\liminf_{r\to 0} |u(r)|=0$. Let
	$$
	\Omega(r) := \int_0^r W(t) t^2 \,dt \,,
	\qquad
	C:= \sup_{r>0} r^{-1} \Omega(r) \,.
	$$
	We may also assume that $u'\in L^2(0,\infty)$ and $C<\infty$, for otherwise there is nothing to prove. For any $0<\epsilon<M<\infty$ we have
	\begin{align}\label{eq:directproof2}
		\int_\epsilon^M W(r)|u(r)|^2\,dr & = \int_\epsilon^M \Omega'(r)r^{-2}|u(r)|^2\,dr \notag \\
		& = - 2 \re \int_\epsilon^M \Omega(r) r^{-2} \overline{u(r)} \left( u'(r) - r^{-1} u(r) \right)dr \notag \\
		& \quad + \Omega(M) M^{-2} |u(M)|^2 - \Omega(\epsilon) \epsilon^{-2} |u(\epsilon)|^2 \notag \\
		& \leq - 2 \re \int_\epsilon^M \Omega(r) r^{-2} \overline{u(r)} \left( u'(r) - r^{-1} u(r) \right)dr \notag \\
		& \quad + \Omega(M) M^{-2} |u(M)|^2 \notag \\
		& \leq C \left( 2 \int_\epsilon^M \frac{|u(r)|}r \left| u'(r) - r^{-1} u(r) \right| dr + M^{-1} |u(M)|^2\right). 
	\end{align}
	Using \eqref{eq:hardyintro}, we can bound
	\begin{align*}
		2 \! \int_\epsilon^M \! \frac{|u(r)|}r \left| u'(r) - r^{-1} u(r) \right| dr
		& \leq 2 \left( \int_0^\infty \! \frac{|u(r)|^2}{r^2}\,dr \right)^{\!1/2} \! \left( \int_\epsilon^M \! |u'(r)-r^{-1} u(r)|^2\,dr \right)^{\!1/2} \\
		& \leq 4 \left( \int_0^\infty \! |u'(r)|^2\,dr \right)^{\!1/2} \! \left( \int_\epsilon^M \! |u'(r)-r^{-1} u(r)|^2\,dr \right)^{\!1/2}\!\!.
	\end{align*}
	At this point we slightly deviate from the previous proof and we notice that, since
	$$
	|u'(r)-r^{-1} u(r)|^2 = |u'(r)|^2 - r^{-1} (|u|^2)'(r) + r^{-2} |u(r)|^2 \,,
	$$
	we have, integrating by parts,
	\begin{align*}
		\int_\epsilon^M |u'(r)-r^{-1} u(r)|^2\,dr & = \int_\epsilon^M |u'(r)|^2 \,dr - M^{-1} |u(M)|^2 + \epsilon^{-1} |u(\epsilon)|^2\\
		& \leq \int_\epsilon^M |u'(r)|^2 \,dr + \epsilon^{-1} |u(\epsilon)|^2 \,.
	\end{align*}
	The term $\epsilon^{-1} |u(\epsilon)|^2$ can be dealt with in the same way as in the previous proof.
	
	Inserting all this into \eqref{eq:directproof2}, we see that it remains to control $M^{-1} |u(M)|^2$. Since, by \eqref{eq:hardyintro}, $r^{-1}|u(r)|^2$ is integrable with respect to the measure $r^{-1}\,dr$, which has infinite integral near infinity, we must have $\liminf_{r\to\infty} r^{-1} |u(r)|^2 = 0$. Taking a sequence of $M$'s along which $r^{-1}|u(r)|^2$ tends to zero, we deduce the claimed inequality.
\end{proof}


\section{Hardy inequalities on subintervals of the halfline}\label{sec:intervals}

In this section we record doubly weighted Hardy inequalities on intervals of the form $(0,R)$ or $(R,\infty)$. They are rather straightforward consequences of Theorems \ref{weightd} and \ref{weightdinv}. We state them here explicitly because it is those inequalities that will play a role in our application to Schr\"odinger operators in the next section.

\begin{corollary}\label{weightdint1}
	Let $1<p<\infty$, let $R\in(0,\infty)$ and let $V,W$ be nonnegative, a.e.-finite, measurable functions on $(0,R)$ such that
	$$
	\int_0^s V(t)^{-\frac{1}{p-1}}\,dt <\infty 
	\qquad\text{for all}\ s\in(0,R) \,.
	$$
	Then, for any locally absolutely continuous function $u$ on $(0,R)$ with $\liminf_{r\to 0} |u(r)|\!=0$,
	\begin{equation}
		\label{eq:hardyintromainweightd1int1}
		\int_0^R W(r) |u(r)|^p \,dr \leq \frac{p^p}{(p-1)^{p-1}} 
		\overline B_R \int_0^R V(r) |u'(r)|^p\,dr
	\end{equation}
	and
	\begin{equation}
		\label{eq:hardyintromainweightd2int1}
		\int_0^R W(r) |u(r)|^p \,dr \leq \left(\frac{p}{p-1} \right)^p \underline B_R \int_0^R V(r) |u'(r)|^p\,dr
	\end{equation}
	with
	\begin{equation*}
		\overline B_R := \sup_{0<s<R} \left( \int_0^s V(t)^{-\frac1{p-1}}\,dt \right)^{p-1} \left( \int_s^R W(t)\,dt \right)
	\end{equation*}
	and
	\begin{equation*}
		\underline B_R := \sup_{0<s<R} \left( \int_0^s V(t)^{-\frac{1}{p-1}}dt \right)^{-1} \int_0^s W(t) \left( \int_0^t V(t')^{-\frac{1}{p-1}}dt' \right)^p dt \,.
	\end{equation*}
\end{corollary}

\begin{proof}
	If $\int_0^R V(t)^{-\frac{1}{p-1}}\,dt =\infty$, the corollary follows as in the proof of Lemma \ref{cov} from \eqref{eq:hardyintromain1}. Thus, assume that the integral is finite. We extend the functions $W,V,u$ to $(R,\infty)$ as follows. We extend $W$ by zero, $V$ in an arbitrary manner such that the integral condition in Theorem \ref{weightd} is satisfied and $u$ by the constant $u(R)$. (Note that $\lim_{r\to R}u(r)$ exists if $\int_0^R V(r)|u'(r)|^p\,dr<\infty$. This is standard for $V\equiv 1$ and follows for general $V$ as in the corollary by the argument in the proof of Lemma \ref{cov}.) The corollary now follows from Theorem \ref{weightd}, noting that the suprema $\overline B$ and $\underline B$ for the extended functions can be restricted to $s<R$.
\end{proof}

\begin{corollary}\label{weightdint2}
	Let $1<p<\infty$, let $R\in(0,\infty)$ and let $V,W$ be nonnegative, a.e.-finite, measurable functions on $(R,\infty)$ such that
	$$
	\int_R^s V(t)^{-\frac{1}{p-1}}\,dt <\infty
	\qquad\text{for all}\ s\in(R,\infty) \,.
	$$
	Then, for any locally absolutely continuous function $u$ on $(R,\infty)$ and any $M\in(0,\infty)$,
	\begin{equation}
		\label{eq:hardyintromainweightd1int2}
		\int_R^\infty W(r) |u(r)|^p \,dr \leq \frac{p^p}{(p-1)^{p-1}} 
		\overline B^R(M) \left( \int_R^\infty V(r) |u'(r)|^p\,dr + M^{-p+1} |u(R)|^p \right)
	\end{equation}
	and
	\begin{equation}
		\label{eq:hardyintromainweightd2int2}
		\int_R^\infty W(r) |u(r)|^p \,dr \leq \left(\frac{p}{p-1} \right)^p \underline B^R(M) \left( \int_R^\infty V(r) |u'(r)|^p\,dr + M^{-p+1} |u(R)|^p \right)
	\end{equation}
	with
	\begin{equation*}
		\overline B^R(M) := \sup_{s>R} \left( M+ \int_R^s V(t)^{-\frac1{p-1}}\,dt \right)^{p-1} \left( \int_s^\infty W(t)\,dt \right)
	\end{equation*}
	and
	\begin{equation*}
		\underline B^R(M) := \sup_{s>R} \left( M+ \int_0^s V(t)^{-\frac{1}{p-1}}dt \right)^{-1} \int_0^s W(t) \left( M+ \int_R^t V(t')^{-\frac{1}{p-1}}dt' \right)^p dt \,.
	\end{equation*}
\end{corollary}

\begin{proof}
	We extend the functions $W,V,u$ to $(0,R)$ as follows. We extend $W$ by zero, $V$ in an arbitrary manner such that
	$$
	\int_0^R V(t)^{-\frac{1}{p-1}}\,dt = M
	$$
	and $u$ by setting
	$$
	u(r) := u(R) M^{-1} \int_0^r V(t)^{-\frac{1}{p-1}}\,dt
	\qquad\text{for all}\ r\in(0,R) \,.
	$$
	For the existence of $u(R)$ see the preceding proof. The corollary now follows from Theorem \ref{weightd}, noting that the suprema $\overline B$ and $\underline B$ for the extended functions can be restricted to $s>R$.	
\end{proof}

The following two corollaries follow from Theorem \ref{weightdinv} in the same way as the previous two corollaries follow from Theorem \ref{weightd}. Alternatively, they follow from the previous two corollaries by the same change of variables $r\mapsto r^{-1}$ that was used to deduce Theorem \ref{weightdinv} from Theorem \ref{weightd}.

\begin{corollary}\label{weightdinvint1}
	Let $1<p<\infty$, let $R\in(0,\infty)$ and let $V,W$ be nonnegative, a.e.-finite, measurable functions on $(0,R)$ such that
	$$
	\int_s^R V(t)^{-\frac{1}{p-1}}dt <\infty
	\qquad\text{for all}\ s\in (0,R) \,.
	$$
	Then for any locally absolutely continuous function $u$ on $(0,R)$ and any $M\in(0,\infty)$,
	\begin{equation}
		\label{eq:hardyintromainweightdinv1int1}
		\int_0^R W(r) |u(r)|^p \,dr \leq \frac{p^p}{(p-1)^{p-1}} 
		\overline B'_R(M) \left( \int_0^R V(r) |u'(r)|^p\,dr + M^{-p+1} |u(R)|^p \right)
	\end{equation}
	and
	\begin{equation}
		\label{eq:hardyintromainweightdinv2int1}
		\int_0^R W(r) |u(r)|^p \,dr \leq \left(\frac{p}{p-1} \right)^p \underline B'_R(M) \left( \int_0^R V(r) |u'(r)|^p\,dr + M^{-p+1} |u(R)|^p \right)
	\end{equation}
	with
	\begin{equation*}
		\overline B'_R(M) := \sup_{0<s<R} \left( M + \int_s^R V(t)^{-\frac1{p-1}}\,dt \right)^{p-1} \left( \int_0^s W(t)\,dt \right)
	\end{equation*}
	and
	\begin{equation*}
		\underline B'_R(M) := \sup_{0<s<R} \left( M+ \int_s^R V(t)^{-\frac{1}{p-1}}dt \right)^{-1} \int_s^\infty W(t) \left( M+ \int_t^R V(t')^{-\frac{1}{p-1}}dt' \right)^p dt \,.
	\end{equation*}
\end{corollary}

\begin{corollary}\label{weightdinvint2}
	Let $1<p<\infty$, let $R\in(0,\infty)$ and let $V,W$ be nonnegative, a.e.-finite, measurable functions on $(R,\infty)$ such that
	$$
	\int_s^\infty V(t)^{-\frac{1}{p-1}}dt <\infty
	\qquad\text{for all}\ s\in (R,\infty) \,.
	$$
	Then for any locally absolutely continuous function $u$ on $(R,\infty)\!$ with $\liminf_{r\to \infty} \!|u(r)|\!=0$,
	\begin{equation}
		\label{eq:hardyintromainweightdinv1int2}
		\int_R^\infty W(r) |u(r)|^p \,dr \leq \frac{p^p}{(p-1)^{p-1}} 
		\overline B'^R \int_R^\infty V(r) |u'(r)|^p\,dr
	\end{equation}
	and
	\begin{equation}
		\label{eq:hardyintromainweightdinv2int2}
		\int_R^\infty W(r) |u(r)|^p \,dr \leq \left(\frac{p}{p-1} \right)^p \underline B'^R \int_R^\infty V(r) |u'(r)|^p\,dr
	\end{equation}
	with
	\begin{equation*}
		\overline B'^R := \sup_{s>R} \left( \int_s^\infty V(t)^{-\frac1{p-1}}\,dt \right)^{p-1} \left( \int_R^s W(t)\,dt \right)
	\end{equation*}
	and
	\begin{equation*}
		\underline B'^R := \sup_{s>R} \left( \int_s^\infty V(t)^{-\frac{1}{p-1}}dt \right)^{-1} \int_s^\infty W(t) \left( \int_t^\infty V(t')^{-\frac{1}{p-1}}dt' \right)^p dt \,.
	\end{equation*}
\end{corollary}


\section{An application to the theory of Schr\"odinger operators}\label{sec:so}

In this section we explain the use of inequalities \eqref{eq:hardyintromainweightd1} and \eqref{eq:hardyintromainweightd2} for a problem concerning the eigenvalues of Schr\"odinger operators. We wish to stress that \emph{both} inequalities \eqref{eq:hardyintromainweightd1} and \eqref{eq:hardyintromainweightd2} are equally important, the former when the underlying space dimension is one and the latter when it is three or higher. The two-dimensional case is somewhat special, but it is eventually also based on \eqref{eq:hardyintromainweightd1}.

\begin{theorem}\label{finitelymanybirman}
	Let $Q\in L^1_\loc(\R^d)$ with $Q_+\in L^p(\R^d)+L^\infty(\R^d)$, where $p=1$ if $d=1$, $p>1$ if $d=2$ and $p= d/2$ if $d\geq 3$. In addition, assume that there is an $R<\infty$ such that for all $r\geq R$,
	\begin{equation}\label{eq:finitelymanybirman}
		\begin{aligned}
			\sup_{\omega\in\Sph^{d-1}} \int_r^\infty Q(s\omega)_+\, s^{-|d-2|+1} \,ds \leq \frac{|d-2|}{4\,r^{|d-2|}} \qquad \text{if}\ d\neq 2 \,, \\
			\sup_{\omega\in\Sph^1} \int_r^\infty Q(s\omega)_+\, s\,ds \leq \frac{1}{4\,\ln r} \qquad \text{if}\ d=2 \,.
		\end{aligned}
	\end{equation}
	Then the negative spectrum of $-\Delta-Q$ consists of at most finitely many eigenvalues.
\end{theorem}

This theorem appears as \cite[Proposition 4.19]{FrLaWe}. We note that by \cite[Proposition 4.3]{FrLaWe} under the conditions in the first sentence, $-\Delta-Q$ can be defined as a selfadjoint operator in $L^2(\R^d)$ via a lower semibounded and closed quadratic form. We also emphasize that by `finitely many eigenvalues' we mean, more precisely, that the total spectral multiplicity of the negative spectrum is finite.

As a consequence of Theorem \ref{finitelymanybirman}, we see that the finiteness of the negative spectrum follows from the existence of an $R<\infty$ such that
\begin{equation}
	\label{eq:finitelymany}
	Q(x) \leq
	\begin{cases}
		\frac{(d-2)^2}{4|x|^2} & \text{if}\ d\neq 2 \,,\\
		\frac{1}{4|x|^2(\ln|x|)^2} & \text{if}\ d=2 \,,
	\end{cases}
	\qquad\text{for all}\ |x|\geq R \,.
\end{equation}
The constants here are sharp in the sense that, if there are $R<\infty$ and $\epsilon>0$ such that
\begin{equation}
	\label{eq:infinitelymany}
	Q(x) \geq
	\begin{cases}
	(1+\epsilon)\, \frac{(d-2)^2}{4|x|^2} & \text{if}\ d\neq 2 \,,\\
	(1+\epsilon)\, \frac{1}{4|x|^2(\ln|x|)^2} & \text{if}\ d=2 \,,
	\end{cases}
	\qquad\text{for all}\ |x|\geq R \,,
\end{equation}
then the negative spectral subspace of $-\Delta-Q$ is infinite dimensional; see \cite[Proposition 4.21]{FrLaWe}.

Conditions of the type \eqref{eq:finitelymany} and \eqref{eq:infinitelymany} go back at least to the work of Courant and Hilbert \cite[Section VI.5]{CoHi}; see also \cite[Theorem IV.6]{Gl}. Our interest was in the more precise condition in Theorem \ref{finitelymanybirman}, which involves an integral bound rather than a pointwise bound. Such integral conditions appear in Birman's paper \cite{Bi}, which contains a qualitative version of Theorem \ref{finitelymanybirman}.

\begin{proof}
	A complete proof of Theorem \ref{finitelymanybirman} is contained in \cite{FrLaWe}. Here we only sketch the major steps, emphasizing the connection with Hardy inequalities.
	
	The first step is to separate the inside from the outside. By this we mean that we will bound $-\Delta-Q$ from below by a direct sum of two operators, one acting in $L^2(B_R)$ and one in $L^2(\overline{B_R}^c)$. Both operators will act as $-\Delta-Q$. In dimensions $d\geq 3$ they will both have Neumann boundary conditions, while in dimensions $d=1,2$ they will have certain Robin boundary conditions. It is standard that the operator in $L^2(B_R)$ will have discrete spectrum and so, in particular, the number of its negative eigenvalues is finite. Thus, it remains to prove that the number of negative eigenvalues of the operator in $L^2(\overline{B_R}^c)$ is finite. In fact, we will show that this operator is nonnegative, that is, we will show that
	$$
	\int_{\overline{B_R}^c} Q(x)|\psi(x)|^2\,dx \leq \int_{\overline{B_R}^c} |\nabla\psi(x)|^2\,dx + c_R \int_{\partial B_R} |\psi(x)|^2\,d\sigma(x)
	\quad\text{for all}\ \psi\in H^1(\overline{B_R}^c) \,.
	$$
	Here $c_R=R^{-1}$ if $d=1$, $c_R = (\ln R)^{-1}$ if $d=2$ and $c_R=	0$ if $d\geq 3$. This is the parameter defining the Robin boundary conditions. By Fubini's theorem, the latter inequality will be a consequence of the inequality
	\begin{equation}
		\label{eq:soproof}
		\int_R^\infty Q(r\omega)|u(r)|^2 r^{d-1}\,dr \leq \int_R^\infty |u'(r)|^2 r^{d-1}\,dr + c_R |u(R)|^2
		\qquad\text{for all}\ u\in H^1_d(R,\infty) \,,
	\end{equation}
	where $H^1_d(R,\infty)$ denotes the space of all weakly differentiable functions on $(R,\infty)$ that together with their derivatives belong to $L^2((R,\infty),r^{d-1}\,dr)$. Note that for such functions $u(R)$ is well-defined.
	
	Inequality \eqref{eq:soproof} for $d=1$ follows from \eqref{eq:hardyintromainweightd1int2} with $W(r)=Q(r\omega)$, $V(r)=1$ and $M=R$. Note that by assumption $\overline B^R(M)\leq 1/4$. Inequality \eqref{eq:soproof} for $d\geq 3$ follows from \eqref{eq:hardyintromainweightdinv2int2} with $W(r)=Q(r\omega)r^{d-2}$ and $V(r)=r^{d-2}$. Note that by assumption $\overline B'^R\leq 1/4$.
	
	To prove inequality \eqref{eq:soproof} for $d=2$ we may assume that $R\geq 1$. Setting $u(x):= \tilde u(e^x)$, $w(x)=e^{2x}W(e^x)$ and $X=\ln R$, we have
	$$
	\int_R^\infty |u'(r)|^2r\,dr + c_R |u(R)|^2 = \int_X^\infty |f'(x)|^2\,dx + X^{-1} |f(X)|^2
	$$
	and	
	$$
	\int_R^\infty Q(r\omega)|u(r)|^2 r\,dr = \int_X^\infty w(x) |f(x)|^2\,dx \,.
	$$
	Thus, the claimed inequality follows, like the one-dimensional inequality, from \eqref{eq:hardyintromainweightd1int2}.
	
	This completes our sketch of proof of Theorem \ref{finitelymanybirman}.
\end{proof}

We end this paper with a variation of Theorem \ref{finitelymanybirman}, which seems to be new.

\begin{theorem}\label{finitelymanybirman2}
	Let $Q\in L^1_\loc(\R^d)$ with $Q_+\in L^p(\R^d)+L^\infty(\R^d)$, where $p=1$ if $d=1$, $p>1$ if $d=2$ and $p= d/2$ if $d\geq 3$. In addition, assume that there is an $R<\infty$ such that for all $r\geq R$,
	\begin{equation}\label{eq:finitelymanybirman2}
		\begin{aligned}
			\sup_{\omega\in\Sph^{d-1}} \int_r^\infty \left( Q(s\omega) - \frac{(d-2)^2}{4s^2}\right)_+ s \,ds \leq \frac{1}{4\,\ln r} \,.
		\end{aligned}
	\end{equation}
	Then the negative spectrum of $-\Delta-Q$ consists of at most finitely many eigenvalues.
\end{theorem}

For $d=2$, this is Theorem \ref{finitelymanybirman}. We do not know whether, for general $d$, Theorem \ref{finitelymanybirman2} implies Theorem \ref{finitelymanybirman}, but it does imply finiteness of the negative spectrum under the condition that there is an $R<\infty$ such that
$$
Q(x) \leq \frac{(d-2)^2}{4|x|^2} + \frac{1}{4|x|^2(\ln|x|)^2} 
\qquad\text{for all}\ |x|\geq R \,.
$$
This improves the condition \eqref{eq:finitelymany} for $d\neq 2$. Moreover, the condition is optimal in the sense that, if there are $R<\infty$ and $\epsilon>0$ such that
$$
Q(x) \geq \frac{(d-2)^2}{4|x|^2} + (1+\epsilon) \frac{1}{4|x|^2(\ln|x|)^2}
\qquad\text{for all}\ |x|\geq R \,,
$$
then the negative spectral subspace of $-\Delta-Q$ is infinite dimensional. This is proved by an argument similar to that used in \cite[Proposition 4.21]{FrLaWe}, together with the same change of variables $\tilde\psi(x) = |x|^\frac{d-2}{2}\psi(x)$ as in the following proof.

\begin{proof}
	The proof is similar to that of Theorem \ref{finitelymanybirman}. Again, it suffices to prove inequality \eqref{eq:soproof}, but this time we will choose $c_R := \left( (\ln R)^{-1} - \frac{d-2}{2}\right) R^{d-2}$. (In dimensions $d\geq 3$ we could also assume that $R\geq e^{2/(d-2)}$ and choose $c_R=0$.) Writing $\tilde u(r) = r^{\frac{d-2}{2}} u(r)$ and noting that
	$$
	\int_R^\infty |u'(r)|^2r^{d-1}\,dr = \int_R^\infty |\tilde u'(r)|^2r \,dr + \int_R^\infty\tfrac{(d-2)^2}{4r^2} |\tilde u(r)|^2 r \,dr + \tfrac{d-2}2 |\tilde u(R)|^2 \,,
	$$
	we see that \eqref{eq:soproof} is equivalent to
	\begin{align*}
		\int_R^\infty \left( Q(r\omega)-\tfrac{(d-2)^2}{4r^2} \right)|\tilde u(r)|^2 r \,dr 
		& \leq \int_R^\infty |\tilde u'(r)|^2r \,dr + (\ln R)^{-1} |\tilde u(R)|^2 \\
		& \quad\ \text{for all}\ \tilde u\in H_2^1(R,\infty) \,.
	\end{align*}
	The latter inequality is precisely the one proved in the proof of Theorem \ref{finitelymanybirman} for $d=2$. This concludes the proof.
\end{proof}


\bibliographystyle{amsalpha}

\begin{thebibliography}{16}
	
\bibitem{BeSh} C. Bennett, R. Sharpley, \textit{Interpolation of operators}. Pure and Applied Mathematics, 129. Academic Press, Inc., Boston, MA, 1988.

\bibitem{Bi} M. \v{S}. Birman, \textit{On the spectrum of singular boundary-value problems}. Mat. Sb. (N.S.) \textbf{55} (97) (1961), 125--174.

\bibitem{CoHi} R. Courant, D. Hilbert, \textit{Methods of mathematical physics. Vol. I}.  Wiley Classics Library. A Wiley-Interscience Publication. John Wiley \& Sons, Inc., New York, 1989.

\bibitem{Da} E. B. Davies, \textit{A review of Hardy inequalities}. In: The Maz'ya anniversary collection, Vol. 2 (Rostock, 1998), 55--67, Oper. Theory Adv. Appl., \textbf{110}, Birkh\"auser, Basel, 1999.

\bibitem{FrLaWe} R. L. Frank, A. Laptev, T. Weidl, {\it Schr\"odinger Operators: Eigenvalues and Lieb--Thirring Inequalities}. Cambridge University Press, Cambridge, to appear.	

\bibitem{Gl} I. M. Glazman, \textit{Direct methods of qualitative spectral analysis of singular differential operators}. Daniel Davey \& Co., Inc., New York 1966.

\bibitem{KaKr} I. S. Kac, M. G. Kre\u{\i}n, \textit{Criteria for the discreteness of the spectrum of a singular string}. Izv. Vys\v{s}. U\v{c}ebn. Zaved. Matematika (1958), no. 2 (3), 136--153.  

\bibitem{KuMaPe} A. Kufner, L. Maligranda, L.-E. Persson, \textit{The prehistory of the Hardy inequality}. Amer. Math. Monthly, \textbf{113} (2006), no. 8, 715--732.

\bibitem{Le} G. Leoni, \textit{A first course in Sobolev spaces}. Second edition. Graduate Studies in Mathematics, 181. American Mathematical Society, Providence, RI, 2017.

\bibitem{LuZa} W. A. J. Luxemburg, A. C. Zaanen, \textit{Some examples of normed K\"othe spaces}. Math. Ann. \textbf{162} (1965/66), 337--350.

\bibitem{Ma} V. Maz'ya, \textit{Sobolev spaces with applications to elliptic partial differential equations}. Second, revised and augmented edition. Grundlehren der mathematischen Wissenschaften \textbf{342}. Springer, Heidelberg, 2011.

\bibitem{Mu} B. Muckenhoupt, \textit{Hardy's inequality with weights}. Studia Math. \textbf{44} (1972), 31--38.

\bibitem{OpKu} B. Opic, A. Kufner, \textit{Hardy-type inequalities}. Pitman Research Notes in Mathematics Series, 219. Longman Scientific \& Technical, Harlow, 1990.

\bibitem{Ta} G. Talenti, \textit{Osservazioni sopra una classe di disuguaglianze}. Rend. Sem. Mat. Fis. Milano \textbf{39} (1969), 171--185. 

\bibitem{To} G. Tomaselli, \textit{A class of inequalities}. Boll. Un. Mat. Ital. (4) \textbf{2} (1969), 622--631.
	
\bibitem{Za} A. C. Zaanen, \textit{Integration}. Interscience Publishers John Wiley \& Sons, Inc., New York, 1967.

\end{thebibliography}

\end{document}